\documentclass[10pt,leqno]{article}
\usepackage[toc,page]{appendix}
\usepackage{amsmath,amssymb,amsfonts,graphicx,bm,amsthm}
\usepackage[usenames]{xcolor}
\usepackage{soul}
\usepackage[colorlinks=true]{hyperref}
\usepackage{cleveref}
\usepackage{appendix}
\usepackage{pstricks}

\usepackage{enumitem}
\usepackage{latexsym}
\usepackage{stmaryrd}
\usepackage{pstricks}
\usepackage{aliascnt}

\numberwithin{equation}{section} 
\newtheorem{theorem}{Theorem}[section]
\newaliascnt{lemma}{theorem}  
\newtheorem{lemma}[lemma]{Lemma}  
\aliascntresetthe{lemma}  
  
\newaliascnt{definition}{theorem}  
\newtheorem{definition}[definition]{Definition}  
\aliascntresetthe{definition}  
 
\newaliascnt{corollary}{theorem}  
\newtheorem{corollary}[corollary]{Corollary}  
\aliascntresetthe{corollary}  
 
\newaliascnt{proposition}{theorem}  
\newtheorem{proposition}[proposition]{Proposition}  
\aliascntresetthe{proposition}  
 
\newaliascnt{remark}{theorem}  
\newtheorem{remark}[remark]{Remark}  
\aliascntresetthe{remark}  
 
\newaliascnt{notation}{theorem}  
  
\aliascntresetthe{notation}  
 
\newaliascnt{example}{theorem}  
  
\aliascntresetthe{example}  
 
\newaliascnt{conjecture}{theorem}  
  
\aliascntresetthe{conjecture}  
 
\newaliascnt{question}{theorem}  
  
\aliascntresetthe{question}  
 
\newaliascnt{fact}{theorem}  
  
\aliascntresetthe{fact}  
 
\newaliascnt{claim}{theorem}  
  
\aliascntresetthe{claim}  
 
\newcommand{\Boxdot}{\,
                    \setlength{\unitlength}{1ex}
                    \begin{picture}(2,2)
                    \put(0,0){$\Box$}
                    \put(.50,.65){.}
                    \end{picture}
                    \!}

\def\IQC{\hbox{\sf IQC}{} }
\def\CQC{\hbox{\sf CQC}{} }
\def\IPC{\hbox{\sf IPC} }
\def\HA{\hbox{\sf HA}{} }

\def\PA{\hbox{\sf PA}{} }
\def\LC{\hbox{\sf iGLC}{}}
\def\LCS{\hbox{\sf iH}^*_\sigma{}}

\def\CP{\hbox{\sf CP}{} }
\def\TP{\hbox{\sf TP}{} }
\def\NNIL{\hbox{\sf NNIL}{} }
\def\TNNIL{\hbox{\sf TNNIL}{} }

\def\K4{\hbox{\sf K4}{} }
\def\GL{\hbox{\sf GL}{} }
\def\L{\hbox{\sf L}{} }

\def\iGL{\hbox{i\sf GL}{} }

\def\PL{\mathcal{PL}}
\def\PLS{\mathcal{PL}_{\sigma}}
\def\ar{\mid\!\sim}

\def\lle{\hbox{\sf{LLe}}{}^+}
\def\LLE{\sf{LLe}}
\def\lles{\sf{iH}_{\sigma}}

\newcommand{\brt}{\mathrel{\mbox{\textcolor{gray}{$\blacktriangleright$}}}}
\newcommand{\blt}{\mathrel{\mbox{\textcolor{gray}{$\blacktriangleleft$}}}}

\newcommand{\ra }{\rightarrow }

\newcommand{\lr }{\leftrightarrow}

\newcommand{\bo }{\Box }
\newcommand{\blrt }{\brt\!\blt }

\newcommand{\gnumber}[1]{\ulcorner#1\urcorner}

\newcommand{\tinysub}[1]{{^{_{#1}}}}

\begin{document}
\setstcolor{red}
\title{The $\Sigma_1$-Provability Logic of $\HA^*$}
\author{\begin{tabular}{c c}
		Mohammad Ardeshir\thanks{mardeshir@sharif.ir}  
		\quad &	Mojtaba Mojtahedi\thanks{\url{http://mmojtahedi.ir/}}\\
		Department of Mathematical Sciences  
		\quad &Department of Mathematics, \\
		Sharif University of Technology
		\quad &  Statistics and Computer Science, \\ 
		Tehran, Iran
		\quad &College of Sciences,  University of Tehran
\end{tabular}}

\maketitle

\begin{abstract}
For the Heyting Arithmetic $\HA$,  $\HA^*$ is defined
 \cite{VisserThes,Visser82} 
as the theory 
$\{A\mid \HA\vdash A^{\Box}\}$, where $A^{\Box}$ is called the box translation of 
$A$ (Definition \ref{Definition-Box translation}).  We characterize the $\Sigma_1$-provability logic 
of $\HA^*$ as a  
modal theory $\LCS$ (Definition \ref{definition of the main theory}).
\end{abstract}

\tableofcontents

\section{Introduction}\label{sec-introduction}
This paper is a sequel of our previous paper 
\cite{Sigma.Prov.HA}, in which we characterized 
the $\Sigma_1$-provability logic of $\HA$ as a decidable  modal theory $\lles$ 
(see \Cref{definition of the main theory}). 
Most of the materials
of this paper are from the paper mentioned above. Our techniques and proofs are very similar to those 
used there. We use a crucial fact (Theorem \ref{Theorem-Main tool} in this paper) proved in 
\cite{Sigma.Prov.HA}. For the sake of self-containedness as much as possible, we bring here some definitions from that paper.

For an arithmetical  theory $T$ extending  ${\sf HA}$, the following axiom schema is called {\em the Completeness Principle}, ${\sf CP}_\tinysub{T}$:
\begin{center} 
$A\ra\Box_{T} A$.
\end{center}
Recall that {by the work of} G\"{o}del in \cite{Godel}, for each arithmetical formula $A$ and
recursively axiomatizable theory ${T}$ (like {\em Peano Arithmetic} ${\sf PA}$), we can
formalize the statement ``there exists a proof in ${ T}$ for
$A$"  by a sentence of the language of arithmetic, i.e.
${\sf Prov}_\tinysub{T}(\gnumber{A}):=\exists{x}\,{\sf Proof}_\tinysub{T}(x,\gnumber{A})$, 
where $\gnumber{A}$ is the code
of $A$. Now, by {\em interpreting} $\Box_{T}$ by  ${\sf Prov}_\tinysub{T}(\gnumber{A})$, the completeness principle for theory $T$ is read as follows:
\begin{center}
$A\ra {\sf Prov}_\tinysub{T}(\gnumber{A})$.
\end{center}

\noindent Albert Visser in \cite{VisserThes,Visser82} introduced an extension of ${\sf HA}$, 
\begin{center}
$\HA^*:=\HA + {\sf CP}_\tinysub{\sf HA^*}$.
\end{center}
He called $\HA^*$ as a {\em self-completion} of $\HA$. Moreover, he showed that  $\HA^*$ may be  defined as the theory 
$\{A\mid \HA\vdash A^{\Box}\}$, where $A^{\Box}$ is called the {\em box translation} of 
$A$ (Definition \ref{Definition-Box translation}).

The notion of  {\em provability logic} goes back essentially to K.~G\"{o}del 
\cite{Godel33} in 1933. He intended to provide a semantics for 
 Heyting's formalization of {\em intuitionistic logic} $\IPC$.
He defined a
{\em translation}, or {\em interpretation} $\tau$ from the
propositional language to the modal language such that
\begin{center}
${\sf IPC}\vdash A\quad\quad\Longleftrightarrow \quad\quad{\sf S4}\vdash \tau(A)$.
\end{center}

Now the question is whether we can find  some modal propositional theory 
such  that the $\Box$ operator captures  the notion of {\em
provability} in  Peano Arithmetic ${\sf PA}$. Hence 
the question is to find some propositional modal theory $T_\Box$ such that:
$$T_\Box\vdash A\quad\quad \Longleftrightarrow \quad\quad \forall{*} \ \PA\vdash A^*$$
By  $(\ )^*$, we mean a mapping from the modal language to the first-order language of arithmetic, such that
\begin{itemize}
\item $p^*$ is an arithmetical first-order sentence, for any atomic 
variable $p$, and $\bot^*=\bot$,
\item $(A\circ B)^*=A^*\circ B^*$, for $\circ\in\{\vee, \wedge, \rightarrow\}$,
\item $(\Box A)^*:=\exists{x}\,{\sf Proof}_\tinysub{\sf PA}(x,\gnumber{A^*})$.
\end{itemize}

 It turned out that ${\sf S4}$ is {\em not} a right candidate for
interpreting  the notion of {\em provability}, since
$\neg\,\Box\bot$ is a theorem of ${\sf S4}$, contradicting
G\"{o}del's second incompleteness theorem (Peano Arthmetic ${\sf PA}$, does not prove its own consistency).

{Martin L\"{o}b in 1955  showed } \cite{Lob} {that the L\"{o}b's 
rule ($ \Box A\to A/A $)  is valid. Then } in 1976, Robert Solovay \cite{Solovay} proved that  the right modal logic, in
which the $\Box$ operator interprets the notion of {\em
provability in {\sf PA}}, is $\GL$. This modal logic is well-known as the
G\"{o}del-L\"{o}b logic, and has the following axioms and
rules:
\begin{itemize}
\item all tautologies of classical propositional logic,
\item $\Box (A\rightarrow B)\rightarrow(\Box A\rightarrow\Box B)$,
\item $\Box A\rightarrow\Box\Box A$,
\item L\"{o}b's axiom \textup{(}{\sf L}\textup{)}: $\Box(\Box A\rightarrow A)\rightarrow\Box A$,
\item Necessitation Rule: $A/\Box A$,
\item Modus ponens: $(A,A\rightarrow B)/B$.
\end{itemize}
\noindent {\bf Theorem.} {{(Solovay-L\"{o}b)}} 
{\em For any sentence $A$ in the language of modal logic, ${\sf
GL}\vdash A$ if and only if for all interpretations $(\ )^*$,
${\sf PA}\vdash A^*$. }
\vspace{.15cm}

Now let we restrict the map $(\ )^*$ on the atomic variables in the following sense. For any atomic variable 
$p$, $(p)^*$ is a $\Sigma_1$ sentence. This translation or interpretation is called  $\Sigma_1$-interpretation.
On the other hand, let ${\sf GLV} = \GL + {\sf CP}_a$, where 
${\sf CP}_a$ is the completeness principle restricted to atomic variables, i.e., $p\ra\Box p$. Albert Visser  \cite{VisserThes} proved the following result:   
\vspace{.15cm}

\noindent {\bf Theorem.} {{(Visser)}} 
{\em For any sentence $A$ in the language of modal logic, ${\sf
GLV}\vdash A$ if and only if for all $\Sigma_1$ interpretations $(\ )^*$,
${\sf PA}\vdash A^*$. }
\vspace{.15cm}

The question of generalizing Solovay's result from classical
theories to intuitionistic ones, such as the intuitionistic
counterpart of ${\sf PA}$, well-known as 
${\sf HA}$, proved to be remarkably difficult and remains a major {\em open problem} 
since the end of 70s \cite{ArtBekProv}. 
For a detailed history of the origins, backgrounds and motivations of the {\em provability logic}, 
we refer the readers to 
\cite{ArtBekProv}.

The following list contains crucial results about the provability logic of $\HA$ with arithmetical nature:
\begin{itemize}
\item John Myhill 1973 and Harvey Friedman 1975. $ {\sf HA}\nvdash \Box_\tinysub{{\sf HA}} (A\vee B)\to(\Box_\tinysub{{\sf HA}} A\vee \Box_\tinysub{{\sf HA}} B)$, \cite{Myhill,Friedman75}.
\item Daniel Leivant 1975. ${\sf HA}\vdash\Box_\tinysub{{\sf HA}}(A\vee B)\to\Box_\tinysub{{\sf HA}}(\Boxdot_\tinysub{{\sf HA}} A\vee\Boxdot_\tinysub{{\sf HA}} B)$, in which $\Boxdot_\tinysub{{\sf HA}} A$ is a shorthand for
$A\wedge\Box_\tinysub{{\sf HA}} A$, \cite{Leivant-Thesis}. 
\item Albert Visser 1981. ${\sf HA}\vdash\Box_\tinysub{{\sf HA}}\neg\neg\,\Box_\tinysub{{\sf HA}} A\to\Box_\tinysub{{\sf HA}}\Box_\tinysub{{\sf HA}} A$ and 
${\sf HA}\vdash \Box_\tinysub{{\sf HA}}(\neg\neg\,\Box_\tinysub{{\sf HA}} A\to\Box_\tinysub{{\sf HA}} A)\to\Box_\tinysub{{\sf HA}}(\Box_\tinysub{{\sf HA}} A\vee \neg\,\Box_\tinysub{{\sf HA}} A)$, \cite{VisserThes,Visser82}.
\item {Rosalie Iemhoff 2001 introduced} a uniform axiomatization  of all known axiom schemas of  the provability logic of ${\sf HA}$ in an extended language 
with a bimodal operator $\rhd$. In her Ph.D. dissertation \cite{IemhoffT}, Iemhoff raised a conjecture that implies directly that her axiom system, ${\sf iPH}$,  restricted to the normal modal language, is equal to the provability logic of  ${\sf HA}$, \cite{IemhoffT}. 
\item {Albert Visser 2002  introduced} a decision algorithm for ${\sf HA}\vdash A$, for all modal propositions $A$ not containing any atomic variable, i.e. $ A $ is made up of $\top, \bot$ via  the unary modal connective $\Box_\tinysub{{\sf HA}}  $ and propositional connectives $ \vee,\wedge,\to $, \cite{Visser02}. 
\item Mohammad Ardeshir and Mojtaba Mojtahedi  2014 characterized  the $\Sigma_1$-provability logic of $\HA$ as a decidable  modal theory \cite{Sigma.Prov.HA}, named there and here as  $\lles$. Recently, Albert Visser and Jetze Zoethout \cite{Vis-Jet-2018arXiv} proved this result by an alternative method.
\end{itemize}

The authors of \cite{reduction} found  a 
{\em reduction} of the Solovay-L\"{o}b Theorem  to the Visser Theorem {\em only by propositional substitutions} \cite{reduction}. This result is tempting to think of applying similar method  for the intuitionistic case. However it seems to us that there is no obvious way of doing  such reduction for the intuitionistic case, and it should be more complicated than the classical case.

In this paper, we introduce an axiomatization of a decidable modal theory $\LCS$ (see Definition 
\ref{definition of the main theory})  
and prove that it is the  
$ \Sigma_1 $-provability logic of $\HA^*$. 
This arithmetical theory is  defined \cite{VisserThes,Visser82} as the theory 
$\{A\mid \HA\vdash A^{\Box}\}$, 
where $A^{\Box}$ is called {\em the box translation} of 
$A$ (Definition \ref{Definition-Box translation}). 
 It is worth mentioning that our proof of the $\Sigma_1$-provability logic of $\HA^*$ is in some sense, a {\em reduction} to the proof of the $\Sigma_1$-provability logic of $\HA$, {\em only by propositional modal logic}.

\section{Definitions, conventions and basic facts}
The propositional non-modal language contains atomic variables,
$\vee, \wedge, \ra, \bot$ and  propositional modal language is
propositional non-modal language plus  $\square$. We use $\boxdot
A$ as a shorthand for $A\wedge\square A$. For simplicity, in this
paper we use propositional language instead of propositional {\em
modal} language. \IPC is the intuitionistic propositional
non-modal logic over usual propositional non-modal language.
$\IPC_\square$ is the same theory \IPC in the extended language
of propositional modal language, i.e. its language is
propositional modal language and its axioms and rules are the
same as the one in $\IPC$. Since we have no axioms for $\square$
in $\IPC_\square$, it is obvious that $\square A$ for each $A$,
behave exactly like an atomic variable inside $\IPC_\square$.
Note that nothing more than  symbol of $A$ plays a role in
$\square A$. The first-order intuitionistic theory is denoted
with $\IQC$ and $\CQC$ is its classical closure, i.e. $\IQC$ plus
the principle of excluded middle. We have the usual first-order
language of arithmetic which has a primitive recursive function
symbol for each primitive recursive function. We use the same
notations and definitions for Heyting's arithmetic \HA as in
\cite{TD}, and Peano Arithmetic \PA is \HA plus the principle of
excluded middle. For a set of sentences and rules
$\Gamma\cup\{A\}$ in propositional non-modal, propositional modal
or first-order language, $\Gamma\vdash A$ means that $A$ is
derivable from $\Gamma$ in the system $\IPC, \IPC_\square,\IQC$,
respectively. 

\begin{definition}\label{Definition-Arithmetical substitutions}
Suppose $T$ is an r.e arithmetical theory and $\sigma$ is a
function from atomic variables to arithmetical sentences. We
extend $\sigma$ to all modal propositions $A$, inductively:
\begin{itemize}
\item $\sigma_T(A):=\sigma(A)$ for atomic $A$,
\item $\sigma_T$ distributes over $\wedge, \vee, \ra$,
\item $\sigma_T(\square A):=Pr_{T}(\ulcorner\sigma_T(A)\urcorner)$, in
which $Pr_{T}(x)$ is the $\Sigma_1$-predicate that formalizes
provability of a sentence with G\"{o}del number $x$, in the
theory $T$.
\end{itemize}
We call $\sigma$ to be a $\Sigma_1$-substitution, if for every
atomic $A$, $\sigma(A)$ be a $\Sigma_1$-formula.
\end{definition}

\begin{definition}\label{Definition-Provability Logic}
Provability logic of a sufficiently strong theory, $T$ is defined
to be a modal propositional theory $\PL(T)$ such that
$\PL(T)\vdash A$  iff for arbitrary arithmetical substitution
$\sigma_T$, \ \ $T\vdash\sigma_T(A)$. If we restrict the
substitutions to $\Sigma$-substitutions, then the new modal
theory is $\PLS(T)$.
\end{definition}

\begin{lemma}\label{Lemma-boxed as atomic}
Let $A(p_1, \ldots, p_n)$ be a non-modal proposition with
$p_i\neq p_j$ for all $0<i<j\leq n$. Then for every modal
sentences $B_1, \ldots, B_n$ with $B_i\neq B_j$ for $0<i<j\leq n$
we have:
$$\IPC\vdash A \text{\ \ iff\ \ } \IPC_\square\vdash A[p_1|\square B_1,\ldots,p_n|\square B_n].$$
\end{lemma}
\begin{proof}
By simple inductions on complexity of proofs in \IPC and
$\IPC_\square$.
\end{proof}

The following definition, the Beeson-Visser box-translation, is
essentially from (\cite[Def.4.1]{Visser82}).
\begin{definition}\label{Definition-Box translation}
For every proposition $A$ in  modal propositional language, we
associate a proposition $A^\square$, called  box-translation
of $A$, defined inductively as follows:
\begin{itemize}
\item $A^\square:= A\wedge\square A$, for atomic $A$, and $\bot^\square =
\bot$,
\item $(A\circ B)^\square:=A^\square\circ B^\square$, for $\circ\in\{\vee,\wedge\}$,
\item $(A\ra B)^\square:=(A^\square\ra B^\square)\wedge\square(A^\square\ra B^\square)$,
\item $(\square A)^\square:=\square(A^\square)$.
\end{itemize}
The box-translation can be extended to first-order arithmetical
formulae $A$, as follows:
\begin{itemize}
\item $(\forall{x}A)^\square:=\square(\forall{x}A^\square)\wedge\forall{x}A^\square$,
\item $(\exists{x}A)^\square:=\exists{x}A^\square$.
\end{itemize}
\end{definition}

Define  ${\sf NOI}$ (No Outside Implication) as set of modal
propositions $A$, that any occurrence of $\ra$ is in the scope of
some $\square$. To be able to state an extension of Leivant's
Principle (that is adequate to axiomatize $\Sigma_1$-provability
logic of $\HA$) we need a translation on modal language which we
name it Leivant's translation. We define it recursively as
follows:
\begin{itemize}
\item $A^l:=A$ for atomic $A$, boxed $A$ or $A=\bot$,
\item $(A\wedge B)^l:=A^l\wedge B^l$,
\item $(A\vee B)^l:=\boxdot A^l\vee\boxdot B^l$,
\item $(A\ra B)^l$ is defined by cases: If $A\in {\sf NOI}$, define
$(A\ra B)^l:=A\ra B^l$, else define $(A\ra B)^l:=A\ra B$.
\end{itemize}

\begin{definition} \label{Def-Axiom schema and modal theories}
Minimal provability logic \iGL, is  same as G\"{o}del-L\"{o}b
provability logic \GL, without the principle of excluded middle,
i.e. it has the following axioms and rules:
\begin{itemize}
\item theorem of $\IPC_{\square}$,
\item $\square (A\ra B)\ra(\square A\ra\square B)$,
\item $\square A\ra\square\square A$,
\item L\"{o}b's axiom ({\sf L}): $\square(\square A\ra A)\ra\square A$,
\item Necessitation Rule: $A/\square A$,
\item Modus ponens: $(A,A\ra B)/B$.
\end{itemize}
\subsection{Definition of modal theories}
$i\textbf{K}4$ is $i\GL$ without  L\"{o}b's axiom. Note
that we can get rid of the necessitation rule by adding $\square
A$ to the axioms, for each axiom $A$ in the above list. We will
use this fact later in this paper. 
We list the following axiom schemae:
\begin{itemize}
\item The Completeness Principle: $\CP:=A\ra\square A$.
\item Restricted Completeness Principle to atomic formulae: $\CP_a:=p\ra\square p$, for atomic $p$.
\item Leivant's Principle: ${\sf Le}:=\square(B\vee C)\ra\square (\square B\vee C)$.  \cite{Leivant}
\item Extended Leivant's Principle: ${\sf Le}^+:=\square A\ra\square A^l$.
\item Trace Principle: $\TP:=\square(A\ra B)\ra(A\vee(A\ra B))$. \cite{Visser82}
\end{itemize}
We define theories $\LC:=i\GL+\CP$,  ${\sf H}:=\LC+\TP$,
$\LLE:=i\GL+{\sf Le}$ and $\lle:=i\GL+{\sf Le}^++\CP_a$. Note
that in the presence of \CP and modus ponens, the necessitation
rule is superfluous. Later we will find Kripke semantics for
$\LC$  and also we will see that $\LC$ and  $\lle$ proves the
same formulae of restricted complexity \textup{(}$\TNNIL$\textup{)}.
\end{definition}

\subsection{$\HA^*$ and $\PA^*$}
$\HA^*$ and $\PA^*$ were first introduced in \cite{Visser82}.
These theories are defined as
$$\HA^*:=\{A\mid\HA\vdash A^\square\} \quad \text{ and }\quad
\PA^*:= \{A\mid\PA\vdash A^\square\}.$$

Visser in \cite{Visser82} showed that the provability logic of
$\PA^*$ is {\sf H}, i.e. ${\sf H}\vdash A$ iff  for all
arithmetical substitution $\sigma$, \ $\PA^*\vdash \sigma_{\sf
PA^*}(A)$. That means that $$\PL(\PA^*)=\PLS(\PA^*)={\sf H}.$$

\begin{lemma}\label{Lemma-Properties of Box translation}
\begin{enumerate}
\item For any arithmetical $\Sigma_1$-formula $A$,  $\HA\vdash A\lr{A^\square}$.
\item $\HA$ is closed under the box-translation, i.e., for any arithmetical formula
$A$, $\HA\vdash A$ implies $\HA\vdash A^\square$, so
$\HA\subseteq\HA^*$.
\end{enumerate}
\end{lemma}
\begin{proof}
\begin{enumerate}
\item See \cite{Visser82}(4.6.iii).
\item See \cite{Visser82}(4.14.i).
\end{enumerate}
\end{proof}

\begin{lemma}\label{Lemma-Properties of Box translation 2}
For any $\Sigma_1$-substitution $\sigma$ and each propositional
modal sentence $A$, we have $\HA\vdash \sigma_{{\sf
HA}}(A^\square)\lr(\sigma_{{\sf HA}^*}(A))^\square$ and hence
\begin{center}
$\HA\vdash\sigma_{{\sf HA}}(A^\square)$  \quad iff \quad  $\HA^*\vdash\sigma_{{\sf HA}^*}(A)$
\end{center}
\end{lemma}
\begin{proof}
Use induction on the complexity of $A$. All the steps are
straightforward. For the atomic case, we use Lemma
\ref{Lemma-Properties of Box translation}.1.
\end{proof}

\begin{remark}
{\em This lemma can be combined with the characterization of the 
$\Sigma_1$-provability logic of $\HA$  to derive directly a characterization of the 
$\Sigma_1$-provability logic of $\HA^*$:
\begin{center}{\em
$A$ belongs to the $\Sigma_1$-provability logic of $\HA^*$ iff $A^\Box$ belongs to the
 $\Sigma_1$-provability logic of ${\HA} $.}
\end{center}
This means that we  have a decision algorithm 
for the $ \Sigma_1 $-provability logic of $ \HA^* $. The rest of this paper is devoted to axiomatize the  $\Sigma_1$-provability logic of $\HA^*$.
}
\end{remark}
\section{Propositional modal logics}
\subsection{NNIL formulae and related topics}

The class of {\em No Nested Implications in the Left}, \NNIL
formulae in a propositional language was introduced in \cite{Visser-Benthem-NNIL}
 , and more explored in \cite{Visser02}. The crucial
result of \cite{Visser02} is providing an algorithm that as
input, gets a non-modal proposition $A$ and returns its best \NNIL
approximation $A^*$ from below, i.e., $\IPC\vdash A^*\ra A$ and
for all \NNIL formula $B$ such that $\IPC\vdash B\ra A$, we have
$\IPC\vdash B\ra A^*$. In the following we explain this algorithm
and explain how to extend it to modal propositions.

To define the class of \NNIL propositions, let us first define a
complexity measure $\rho$ on non-modal propositions as follows:
\begin{itemize}
\item $\rho p= \rho\bot=\rho\top = 0$, where $p$ is an atomic proposition, 
\item $\rho(A\wedge B) = \rho(A\vee B) = max (\rho A, \rho B)$,
\item $\rho(A\ra B) = max (\rho A +1, \rho B)$,
\end{itemize}
Then $\NNIL = \{A\mid \rho A\leq 1\}$.
\begin{definition}\label{non-modal complexity}
We define a measure complexity for modal propositions $D$ as
follows:
\begin{itemize}
\item $I(D):=\{E\in Sub(D) \mid E \text{ is an implication that is not in the scope of a } \square\}$,
\item $\mathfrak{i}(D):=max\{|I(E)|\mid E\in I(D)\}$, where $|X|$
is the number of elements of $X$,
\item $\mathfrak{c}D:=$ the number of occurrences of logical connectives which is not in the scope of a $\square$,
\item $\mathfrak{d}D:=$ the maximum number of nested boxes. To be more precise,
\begin{itemize}
\item $\mathfrak{d}D:=0$ for atomic $D$,
\item $\mathfrak{d}D:=max\{\mathfrak{d}D_1,\mathfrak{d}D_2\}$, where $D = D_1\circ D_2$
and $\circ\in\{\wedge, \vee, \ra\}$,
\item $\mathfrak{d}\square D:=\mathfrak{d}D+1$,
\end{itemize}
\item $\mathfrak{o}D:=(\mathfrak{d}D,\mathfrak{i}D,\mathfrak{c}D)$.
\end{itemize}
Note that the measure $\mathfrak{o}D$ is ordered
lexicographically, i.e., $(d,i,c)<(d',i',c')$ iff $d<d'$ or $d=d'
, i<i'$ or $d=d', i=i' , c<c'$.
\end{definition}

\begin{definition}\label{definition-braket-nonmodal}
For any two modal propositions $A,B$, we define  $[A]B$ 
and $[A]'B$, 
 by induction on the complexity of $B$:
\begin{itemize}
\item $[A]p = [A]'p = p$, for atomic $p$, $\top$ and $\bot$,
\item $[A](B_1\circ B_2) = [A](B_1)\circ [A](B_2)$, 
$[A]'(B_1\circ B_2) =[A]'(B_1)\circ [A]'(B_2)$
for $\circ\in\{\vee,\wedge\}$,
\item $[A](B_1\ra B_2) = A\ra (B_1\ra B_2)$, 
$[A]'(B_1\ra B_2)=(A'\wedge B_1)\ra B_2$,
in which $A' = A[{B_1\ra B_2}\mid B_2]$, i.e., replace each
occurrence of $B_1 \ra B_2$ in $A$ by $B_2$,
\end{itemize}
\end{definition}

\subsubsection*{$\NNIL$-algorithm}
 For each proposition $A$,
$A^*$ is produced by induction on complexity measure
$\mathfrak{o}A$ as follows. 
 For details see \cite{Visser02}.
\begin{enumerate}
\item $A$ is atomic, take $A^*:=A$,
\item $ A=B \wedge C$, take $A^*:=B^*\wedge C^*$,
\item $ A=B\vee C$, take $A^*:=B^*\vee C^*$,
\item $ A=B \ra C $, we have several sub-cases. In the following, an occurrence
of $E$ in $D$ is called an {\em outer occurrence}, if $E$ is not
in the scope of an implication.\\
4.a. $C$ contains an outer occurrence of a conjunction. In this
case, we assume some formula $J(q)$ such that
\begin{itemize}
\item $q$ is a propositional variable not occurred in $A$,
\item $q$ is outer in $J$ and occurs exactly once,
\item $C=J[q|(D\wedge E)]$.
\end{itemize}
Such $J$ obviously exists. Now set $C_1:=J[q|D], C_2:=J[q|E]$ and
$A_1:=B\ra C_1, A_2:=B\ra C_2$ and finally, define
$A^*:=A_1^*\wedge A_2^*$.

4.b. $B$ contains an outer occurrence of a disjunction. In this
case we suppose some formula $J(q)$ such that
\begin{itemize}
\item $q$ is a propositional variable not occurred in $A$,
\item $q$ is outer in $J$ and occurs exactly once,
\item $B=J[q|(D\vee E)]$.
\end{itemize}
Such $J$ obviously exists.  Now set $B_1:=J[q|D], B_2:=J[q|E]$
and $A_1:=B_1\ra C, A_2:=B_2\ra C$ and finally, define
$A^*:=A_1^*\wedge A_2^*$.

4.c.  $B=\bigwedge X$ and $C=\bigvee Y$ and $X,Y$ are sets of
implications or atoms. We have several sub-cases:

4.c.i. $X$ contains atomic $p$. Set
$D:=\bigwedge(X\setminus\{p\})$ and take $A^*:=p\ra(D\ra C)^*$.

4.c.ii. $X$ contains $\top$. Define
$D:=\bigwedge(X\setminus\{\top\})$ and take $A^*:=(D\ra C)^*$.

4.c.iii. $X$ contains $\bot$. Take $A^*:=\top$.

4.c.iv. $X$ contains only implications. For any $D=E\ra F\in X$,
let
$$B\downarrow D:=\bigwedge((X\setminus\{D\})\cup\{F\}).$$
Let $Z:=\{E\mid E\ra F\in X\}\cup\{C\}$ and
$A_0:=[B]Z:=\bigvee\{[B]E\mid E\in Z\}$. Now if
$\mathfrak{o}A_0<\mathfrak{o}A$, we take
$$A^*:=\bigwedge\{((B\downarrow D)\ra C)^*|D\in X\}\wedge A_0^*,$$
otherwise, first set $A_1:=[B]'Z$ and then take
$$A^*:=\bigwedge\{((B\downarrow D)\ra C)^*|D\in X\}\wedge A_1^*$$
\end{enumerate}

We can extend $\rho$ to all modal language with $\rho(\square A):=0$.
The class of \NNIL propositions may be defined for propositional
modal language as well, i.e. we call a modal proposition $A$ to be $\NNIL_\square$, if
$\rho(A)\leq 1$ (for extended $\rho$).
We also define two other classes of propositions:
\begin{definition}
$\TNNIL$ \textup{(}Thoroughly $\NNIL$\textup{)} is the
smallest class of propositions such that
\begin{itemize}
\item $\TNNIL$ contains all atomic propositions,
\item if $A, B\in\TNNIL$, then $A\vee B, A\wedge B,\square A\in\TNNIL$,
\item if all $\ra$ occurred in $A$ are contained in the scope of a $\square$ 
\textup{(}or equivalently
$A\in {\sf NOI}$\textup{)} and $A,
B\in\TNNIL$, then $A\ra B\in\TNNIL$.
\end{itemize}
Finally we define $\TNNIL^-$  as the set of all the propositions
like $  A(\square B_1,\ldots,\square B_n)$, such that
$A(p_1,\ldots,p_n)$ is an arbitrary non-modal proposition and
$B_1,\ldots,B_n\in\TNNIL$.
\end{definition}

We can use the same algorithm with slight modifications treating
propositions inside $\square$ as well. First we extend Definition
\ref{definition-braket-nonmodal} to capture modal language.

\begin{definition}\label{definition-braket}
For any two modal propositions $A,B$, we define $[A]B$ 
and $[A]'B$ 
by induction on the complexity of $B$. We extend
Definition {\em \ref{non-modal complexity}} by the following item:
\begin{itemize}
\item $[A]\Box B_1 =[A]'\Box B_1 :=\square B_1$.
\end{itemize}
\end{definition}
It is clear that we are treating a boxed formula as an atomic variable.
\subsubsection*{$\NNIL_{\Box}$-algorithm}
We use $\NNIL$-algorithm with the following changes to
produce a similar $\NNIL$-algorithm for modal language.

\vspace{.1in}

\noindent 1. $A$ is atomic or boxed, take $A^*=A$.

\vspace{.1in}

\noindent 4. An occurrence of $E$ in $D$ is called an {\em outer
occurrence}, if $E$ is neither in the scope of an implication nor
in the scope of a boxed formula.

\vspace{.05in}

\noindent 4. c(i). $X$ contains atomic or boxed formula $p$. We
set $D:=\bigwedge(X\setminus\{p\})$ and take $A^*:=p^*\ra(D\ra
C)^*$.

\begin{remark}\label{rem1}
{\em In fact, we have two ways to find out  $\NNIL_\square$ approximation
of a modal proposition. 

First: simply apply
$\NNIL_{\Box}$-algorithm to a modal proposition $A$ and compute
$A^*$. 

Second: let $B_1,\ldots,B_n$ be all
 boxed sub-formulae of $A$ which are not in the scope of any
other boxes. Let $A'(p_1,\ldots,p_n)$ be  unique non-modal
proposition such that $\{p_i\}_{1\leq i\leq n}$ are fresh atomic
variables not occurred in $A$ and $A=A'[p_1|B_1,\ldots,p_n|B_n]$.
Let $\rho(A):=(A')^*[p_1|B_1,\ldots,p_n|B_n]$. Then it is easy to
observe that $\IPC_{\Box}\vdash\rho(A)\lr A^*$.}
\end{remark}

The above defined algorithm is not deterministic, however by the
following Theorem, we know that $A^*$ is unique up to
$\IPC_{\Box}$ equivalence. The notation
$A\vartriangleright_{\IPC_{\Box},\NNIL_{\Box}}B$ ($A$,
$\NNIL_{\Box}$-preserves $B$) from \cite{Visser02}, means that
for each $\NNIL_{\Box}$ modal proposition $C$, if
$\IPC_{\Box}\vdash C\ra A$, then $\IPC_{\Box}\vdash C\ra B$, in
which $A, B$ are modal propositions.

\begin{theorem}\label{Theorem-NNIL Crucial Properties}
For each modal proposition $A$,
$\NNIL_{\Box}$ algorithm with  input $A$ terminates and
the output formula $A^*$, is an $\NNIL_{\Box}$ proposition such
that $\IPC_{\Box}\vdash A^*\ra A$.
\end{theorem}
\begin{proof}
See \cite[The.~4.5]{Sigma.Prov.HA}.
\end{proof}
\subsubsection*{$\TNNIL$-algorithm}

Here we define $A^+$ as $\TNNIL$-formula approximating $A$.
Informally speaking, to find $A^+$, we first compute $A^*$ and
then replace all outer boxed formula $\square B$ in $A$ by
$\square B^+$. To be more accurate, we  define $A^+$ by induction
on $\mathfrak{d}A$. Suppose that for all $B$ with
$\mathfrak{d}B<\mathfrak{d}A$, we have defined $B^+$. Suppose
that $A'(p_1,\ldots,p_n)$ and $\square B_1,\ldots,\square B_n$
such that $A=A'[p_1|\square B_1,\ldots,p_n|\square B_n]$ where
$A'$ is a non-modal proposition and $p_1,\ldots,p_n $ are fresh
atomic variables (not occurred in $A$). It is clear that
$\mathfrak{d}B_i<\mathfrak{d}A$ and then we can define
$A^+:=(A')^*[p_1|\square B_1^+,\ldots,p_n|\square B_n^+]$.

\begin{lemma}\label{Lemma-NNIL properties}
For any modal proposition $A$,
\begin{enumerate}
\item for all $\Sigma_1$-substitution $\sigma$ we have
$\HA\vdash\Box\sigma_{\sf HA}(A)\lr\Box\sigma_{\sf HA}(A^+)$ and  hence $\HA\vdash\sigma_{\sf HA}(A)$
 iff $\HA\vdash\sigma_{\sf HA}(A^+) $.
\item$i\GL\vdash A_1\ra A_2$ implies $i\GL\vdash A_1^+\ra A_2^+$, and
$i\K4\vdash A_1\ra A_2$ implies $i\K4\vdash A_1^+\ra A_2^+$.
\item $i\GL\vdash A_1\lr A_2$ implies $i\GL\vdash A_1^+\lr A_2^+$, and
$i\K4\vdash A_1\lr A_2$ implies $i\K4\vdash A_1^+\lr A_2^+$.
\end{enumerate}
\end{lemma}
\begin{proof}
See \cite[Corollary~4.8]{Sigma.Prov.HA}.
\end{proof}

\subsubsection*{$\TNNIL^\Box$-algorithm}
\begin{corollary}\label{Corollary HA-NNIL properties}
There exists a $\TNNIL^\Box$-algorithm such that for any modal
proposition $A$, it halts and produces a proposition
$A^-\in\TNNIL^\Box$ such that $\IPC_\square\vdash A^+\ra A^-$.
\end{corollary}
\begin{proof}
Let $A:=B(\square C_1,\ldots,\square C_n)$, and $B(p_1,\ldots,p_n)$ is non-modal.
apparently such $B$ exists. Then define $A^-:=B(\square C_1^+,\ldots,\square C_n^+)$.
Now definition of $A^+$ 
 implies  $A^+=(A^-)^*$ and
 hence   \Cref{Theorem-NNIL Crucial Properties}
 implies that $A^-$ has desired property.
\end{proof}

\begin{lemma}\label{Lemma HA-NNIL properties minus}
For each modal proposition $A$ and  $\Sigma_1$-substitution
$\sigma$,  $\HA\vdash \sigma_{{\sf HA}}A\lr\sigma_{{\sf HA}}A^-$.
\end{lemma}
\begin{proof}
Use definition of $(.)^-$ and \Cref{Lemma-NNIL properties}.1.
\end{proof}

\begin{remark}
{\em Note that $\LC\vdash A\lr B$ does not imply $\LC\vdash
A^+\lr B^+$. A counter-example is $A:=\neg\neg p$ and
$B:=\neg\boxdot(\neg p)$. We have $A^+=A^*=p$ and $\LC\vdash
B^+\lr(\square\neg p\ra p)$. Now one can use Kripke models to
show $\LC\nvdash\neg\neg p\ra(\square\neg p\ra p)$.}
\end{remark}

\begin{remark}\label{Remark-New def for TNNIL algorithm}{\em
In the $\NNIL_{\Box}$-algorithm,
 if we replace the
operation $(\cdot)^*$ by $(\cdot)^\dag$, and change the step 1 to

\noindent 1. $A^\dag:= A$, if $A$ is atomic, and $(\square
B)^\dag:=\square B^\dag$,

\noindent then the new algorithm also halts, and for any modal
proposition $A$, we have $i\K4\vdash A^\dag\lr A^+$.}
\end{remark}

\subsection{Box translation and  propositional theories}

Following Visser's definition of the notion of a {\em base} in
arithmetical theories \cite{Visser82}, we define
\begin{definition}
A modal theory $T$ is called to be {\em closed under
box-translation} if for every proposition $A$, $T\vdash A$ implies
$T\vdash A^\square$.
\end{definition}

\begin{proposition}\label{Proposition-propositional properties of Box translation}
For  arbitrary subset $X$ of $\{\CP,\CP_a,\L\}$, $i\K4+X$ is
closed under  box-translation.
\end{proposition}
\begin{proof}
The proof can be carried out in three steps:

\vspace{.1in}

\noindent 1. For any proposition $A$ first we show that
$\IPC_\square\vdash A$ implies $i\K4\vdash A^\square$. This can
be done by a routine induction on the length of the proof in
$\IPC$. Note that for any axiom $A$ of $\IPC$, we have
$i\K4\vdash A^\square$. As for the rule of modus ponens, suppose
that $\IPC_\square\vdash A$ and $\IPC_\square\vdash A\ra B$. By
induction hypothesis, then $i\K4\vdash A^\square$ and
$i\K4\vdash(A^\square\ra B^\square)\wedge\square(A^\square\ra
B^\square)$ and so $i\K4\vdash B^\square$.

\vspace{.1in}

\noindent 2. Next observe that
$$(\square A\ra\square\square A)^\square=\square A^\square\ra\square\square A^\square$$ and also
$$i\K4\vdash[(\square(A\ra B)\wedge \square A)\ra\square B]^\square\lr
(\square(A^\square\ra B^\square)\wedge \square A^\square)\ra\square B^\square.$$

\vspace{.1in}

\noindent 3. We observe that for any axiom
$A\in X$, $i\K4 + X\vdash A^\square$.

\end{proof}

The following two lemmas will be used in the
proof of 
\Cref{Theorem-NNIL approximation is propositionally equivalent2}.

\begin{lemma}\label{Lemma-property for bracket}
For any modal propositions $A, A'$ and $B$, and any propositional
modal theory $T$ containing $\IPC_\square$,
\begin{enumerate}
\item $i\K4+\square
A^\square\vdash([A]B)^\square\lr([A^\square]B^\square)$.
\item $T\vdash A\lr A'$ implies $T\vdash [A]B\lr [A']B$.
\end{enumerate}
\end{lemma}
\begin{proof}
Proof of both parts are by induction on the complexity of $B$:
\begin{enumerate}
\item The only non-trivial case is when $B$ is an implication. Let
$B:=C\ra D$. By Definitions \ref{definition-braket} and
\ref{Definition-Box translation},
$$([A](C\ra D))^\square =\boxdot(A^\square\ra((C^\square\ra
D^\square)\wedge\square(C^\square\ra D^\square)))$$
and also
$$[A^\square](C\ra D)^\square=(A^\square\ra(C^\square\ra
D^\square))\wedge\square(C^\square\ra D^\square).$$
Now it is easy
to observe  that 
$$i\K4+\square A^\square\vdash([A](C\ra
D))^\square\lr([A^\square](C\ra D)^\square).$$
\item Similar to the first item.
\end{enumerate}
\end{proof}

\noindent\textbf{Notation.} In the sequel of  paper, we use $A\equiv B$
as a shorthand for $i\K4\vdash A\lr B$.

\begin{lemma}\label{Lemma-technical property}
Let $A=B\ra C$ be a modal proposition such that $B=\bigwedge X$ and $C=\bigvee Y$,
where $X$ is a set of implications and $Y$ is a set of
atomic, boxed or implicative propositions. Then
$$
(A^\square)^+\equiv\boxdot\left(  \bigwedge_{E\ra F\in
X}\square\left (\left(E\ra
F\right)^\square\right)^+\ra\left(\bigwedge\left\lbrace\left
(\left(B\downarrow D\ra C\right)^\square\right)^+ \mid D\in
X\right
\rbrace\wedge\left(\left([B]Z\right)^\square\right)^+\right)
\right)$$
where $Z=\{E|E\ra F\in X\}\cup\{C\}$.
\end{lemma}
\begin{proof}
To simplify notations, Let us indicate
\begin{itemize}
\item  the sets  of all atomic and boxed propositions by {\sf At} and {\sf Bo}, respectively,
\item $X':=\{E^\square\ra F^\square\mid E\ra F\in X\}$,
\item $Z':=Z^\square=\{E^\square\mid E\ra F\in X\}\cup\{C^\square\}$,
\item $B':=\bigwedge X'$,
\item   for any $I\subseteq Y$,  $C^I:=\bigvee_{E\ra F\in I}\square(E^\square\ra F^\square)
\vee\bigvee_{E\in I\cap {\sf At}}\square E
\vee\bigvee_{E\ra F\in Y\setminus I}(E^\square\ra F^\square)
\\  \hspace*{1.3in}
\vee\bigvee_{E\in (Y\setminus I)\cap {\sf At}}E
\vee\bigvee_{E\in {\sf Bo}\cap Y} E^\square$,
\item and $Z^I:=\{ E^\square\mid E\ra F\in X\}\cup\{C^I\}$.
\end{itemize}
By repeated use of distributivity of conjunction over disjunction, 
which is valid in \IPC, we have
\begin{equation}\label{Equation-C^box=bigwedge C^I...}
C^\square\equiv\bigwedge_{I\subseteq Y}C^I  \quad \text{and}
\quad Z^\square\equiv\bigwedge_{I\subseteq Y}Z^I
\end{equation}
Note that $A^\square=(B^\square\ra
C^\square)\wedge\square(B^\square\ra C^\square)$, and then by
definition of $(\cdot)^+$,
\begin{center}
$(A^\square)^+ =(B^\square\ra
C^\square)^+\wedge\square(B^\square\ra C^\square)^+$.
\end{center}
Now we compute  the left conjunct:
\begin{align}
\label{Equation-Align1}& \left(B^\square \ra  C^\square \right) ^+=\bigwedge_{I\subseteq Y}\left(B^\square\ra C^I\right) ^+ \\
\label{Equation-Align2}  & \equiv\bigwedge_{I\subseteq Y}\left( \bigwedge_{E\ra F\in X}\square  \left(E^\square\ra F^\square\right)^+\ra\left(\left(\bigwedge_{E\ra F\in X } \left(E^\square\ra F^\square\right) \right)\ra C^I\right)^+\right)   \\
\label{Equation-Align3} & \equiv\bigwedge_{E\ra F\in X}\square\left(\left( E\ra F\right)^\square\right)^+\ra  \bigwedge_{I\subseteq Y}\left( B'\ra C^I\right)^+   \\
\label{Equation-Align4}   &\equiv\bigwedge_{E\ra F\in X}\square\left(\left(E\ra F\right)^\square\right)^+\ra\bigwedge_{I\subseteq Y}\left(\bigwedge\left\lbrace\left(B'\downarrow D'\ra C^I\right)^+ \mid D'\in X'\right\rbrace\wedge\left([B']Z^I\right)^+\right)   \\
\label{Equation-Align5}  & \equiv\bigwedge_{E\ra F\in
X}\square\left(\left(E\ra
F\right)^\square\right)^+\ra\left(\bigwedge\left\lbrace\left(B'\downarrow
D'\ra C^\square\right)^+ \mid D'\in
X'\right\rbrace\wedge\left([B']Z'\right)^+\right)
\end{align}
and hence
\begin{equation}\label{Equation-Align6}
\left(A^\square\right)^+ \equiv  \boxdot\left(\bigwedge_{E\ra
F\in X}\square\left(\left(E\ra
F\right)^\square\right)^+\ra\left(\bigwedge\left\lbrace\left(\left(B\downarrow
D\ra C\right)^\square\right)^+ \mid D\in
X\right\rbrace\wedge\left([B']Z'\right)^+\right)\right)
\end{equation}
Note that \ref{Equation-Align1} and \ref{Equation-Align2} hold by
$\NNIL_\square$-algorithm, \ref{Equation-Align3} holds by properties of
$\IPC_\square$, \ref{Equation-Align4} holds by $\TNNIL$-algorithm,
\ref{Equation-Align5} holds by $\TNNIL$-algorithm and equation
\ref{Equation-C^box=bigwedge C^I...}, and finally equation
\ref{Equation-Align6} is derived from \ref{Equation-Align5} by
deduction in $i\K4$ and $\TNNIL$-algorithm.  Now it is enough to
show that the last formula is equivalent to the following one in
$i\K4$:
\begin{equation}\label{Equation-Align7}
\boxdot\left(\bigwedge_{E\ra F\in X}\square\left(\left(E\ra
F\right)^\square\right)^+\ra\left(\bigwedge\left\lbrace\left(\left(B\downarrow
D\ra C\right)^\square\right)^+ \mid D\in
X\right\rbrace\wedge\left(\left([B]Z\right)^\square\right)^+\right)
\right)
\end{equation}
To show this, it is enough to show
$$i\K4\vdash\bigwedge_{E\ra F\in X}\square\left(\left(E\ra
F\right)^\square\right)^+\ra\left(\left(\left([B]Z\right)^\square\right)^+\lr\left([B']Z'\right)^+\right).$$
Then by \Cref{Lemma-NNIL properties}.2, it is enough
to show $i\K4\vdash\bigwedge_{E\ra F\in X}\square(E\ra
F)^\square\ra(([B]Z)^\square\lr[B']Z')$. Since $\bigwedge_{E\ra
F\in X}\square(E\ra F)^\square \equiv \square B^\square$, then
it is enough to show $i\K4+\square
B^\square\vdash([B]Z)^\square\lr[B']Z'$. Now, by Lemma
\ref{Lemma-property for bracket}.1, we have $i\K4+\square
B^\square\vdash ([B]Z)^\square\lr[B^\square]Z^\square$. Hence we
should show $i\K4+\square
B^\square\vdash[B^\square]Z^\square\lr[B']Z'$. We have
$Z'=Z^\square$ and $i\K4+\square B^\square\vdash B^\square\lr
B'$. Then by Lemma \ref{Lemma-property for bracket}.2,
$i\K4+\square B^\square\vdash[B^\square]Z^\square\lr[B']Z'$.
\end{proof}

\subsection{Axiomatizing $\TNNIL$-algorithm}
In this section, we introduce the 
 axiom set $X$ 
such that $i\K4+X\vdash
(A^\square)^-\lr A^\square$. Note that we may simply choose 
$X:=\{(A^\square)^-\lr A^\square\ |\ A\text{ is arbitrary
proposition} \}$. However,  we want to reduce $X$ to some smaller
efficient set of formulae.

We use some modal variant of Visser's $\brt_\sigma$
in \cite{Visser02}.  It is exactly the same as the relation $ \brt $ in \cite{Sigma.Prov.HA} (sec.~4.3) except for item B2, which is a little bit different: 
\begin{itemize}
\item B$2'$. Let $X$ be a set of implications, $B:=\bigwedge X$ and $A:=B\ra C$.
Also assume that $Z:=\{E | E\ra F\in X\}\cup \{C\}$. Then $A \brt [B]Z $,
\end{itemize}
 The relation $\brt^*$ is defined to be the
smallest relation on modal propositional sentences satisfying:
\begin{itemize}
 \item A1. If $i\K4\vdash A\ra B$, then $A\brt^* B$,
 \item A2. If $A\brt^* B$ and $B\brt^* C$, then $A\brt^* C$,
 \item A3. If $C\brt^* A$ and $C\brt^* B$, then $C\brt^* A\wedge B$,
 \item A4. If $A\brt^* B$, then $\bo A\brt^* \bo B$,
 \item B1. If $A\brt^* C$ and $B\brt^* C$, then $A\vee B\brt^* C$,
 \item B2. Let $X$ be a set of implications, $B:=\bigwedge X$ and $A:=B\ra C$.
 Also assume that $Z:=\{E | E\ra F\in X\}\cup \{C\}$. Then $A\wedge\square B\brt^* [B]Z $,
 \item B3. If $A\brt^* B$, then $p\ra A\brt^* p\ra B$, in which $p$ is atomic or boxed.
\end{itemize}

$A\blrt^* B$ means $A\brt^* B$ and $B\brt^* A$. 
 
\begin{definition}\label{definition of the main theory}
We  define 
\begin{center}
$\LCS:=i\GL+\CP+\{ {\square A\ra\square B}|{ A\brt^* B}\}$. 
\end{center}
\end{definition}

\noindent Note that the $\Sigma_1$-provability logic of $\HA$ is proved  in \cite{Sigma.Prov.HA} to be 
\begin{center} 
 $\lles:=i\GL + \CP_{a}+ {\sf Le}^{+} +   \{{\square A\ra\square B}| {A\brt B}\}$, 
\end{center}
 in which $\CP_{a}$ is
the Completeness Principle restricted to atomic propositions.

\begin{lemma}\label{Lemma-Preservativity is closed under box translation}
For any propositional modal sentences $A, B$, $A\brt^* B$ implies
$A^\square\brt^* B^\square$.
\end{lemma}
\begin{proof}
It is clear that $A\brt^* B$ iff there exists a Hilbert-type
sequence of relations $\{A_i\brt^* B_i\}_{0\leq i\leq n}$ such that
$A_n=A, B_n=B$ and for each $i\leq n$, $A_i\brt^* B_i$ is an
instance of axioms A1 or B2, or it is derived by making use of
some previous members of sequence and some of the rules A2-A4 or
B1 or B3. Hence we are authorized to use induction on the length
of such sequence for $A\brt^* B$ to show $A^\square\brt^* B^\bo$. The
only non-trivial steps are axioms A1 and B2. Suppose that $A\brt^*
B$ is an instance of A1, i.e. $i\K4\vdash A\ra B$. Then by
Proposition \ref{Proposition-propositional properties of Box
translation}, we have $i\K4\vdash A^\square\ra B^\bo$ and hence
again by A1, $A^\square\brt^* B^\square$, as desired.
\\
For treating B2, suppose that $A:=B\ra C$, $B=\bigwedge X$, $X$
is a set of implications and $Z:=\{E | E\ra F\in X\}\cup\{ C\}$.
We must show  $(A\wedge \bo B)^\square\brt^* ([B]Z)^\square$.
Define $X':=\{E^\bo \ra F^\bo | E\ra F\in X\}, B':=\bigwedge X'$.
Hence by  B2, $B'\ra C^\square\wedge\bo B'\brt^* [B']Z^\square$.
Note that we have $\square B'\equiv \square B^\square$ and also
$i\K4+\square B'\vdash B'\lr B^\square$.

Now by using properties of $\brt^*$ (A1-A3) and Lemma
\ref{Lemma-property for bracket}(2), we can deduce $(B^\square\ra
C^\square)\wedge \bo B^\square\brt^* [B^\square]Z^\square$. Then
Lemma \ref{Lemma-property for bracket}(1) implies  $(B^\square\ra
C^\square)\wedge \bo B^\square\brt^* ([B]Z)^\square$. Then by A1
and A2, we can deduce $((B\ra C)\wedge \bo B)^\square\brt^*
([B]Z)^\square$, as desired.
\end{proof}

The following theorem  is analogous to the Theorem 4.18 in \cite{Sigma.Prov.HA}:
\begin{theorem}\label{Theorem-NNIL approximation is propositionally equivalent2}
For any modal proposition $A$, $A^\square \blrt^* (A^\square)^+$.
\end{theorem}

Before proving this theorem, we state a corollary.

\begin{corollary}\label{Corollary-+algorithm}
$\LCS\vdash A^\square\lr (A^\square)^-$.
\end{corollary}
\begin{proof}
Let $A^\square=B(\square C_1,\square C_2,\ldots,\square C_n)$
where $ B(p_1,\ldots,p_n)$ is a non-modal proposition. It isn't
hard to observe that for each $1\leq j\leq n$, $i\K4\vdash
\square C_j\lr \square C_j^\square$. By definition of
$(A^\square)^-$, we have $(A^\square)^-=B(\square
C_1^+,\ldots,\square C_n^+)$. Now by 
\Cref{Lemma-NNIL properties}, we can deduce that $i\K4\vdash
B(\square C_1^+,\ldots,\square C_n^+)\lr B(\square
(C_1^\square)^+,\ldots,\square (C_n^\square)^+)$. Then 
\Cref{Theorem-NNIL approximation is propositionally equivalent2}
implies that $\LCS\vdash \square (C_i^\square)^+\lr \square
C_i^\square$. Hence $\LCS\vdash (A^\square)^-\lr A^\square$.
\end{proof}

\begin{proof}(\textbf{\Cref{Theorem-NNIL approximation is propositionally equivalent2})}
We prove by induction on $\mathfrak{o}(A^\square)$. Suppose that we
have the desired result for each  proposition $B$ with
$\mathfrak{o}(B^\square)<\mathfrak{o}(A^\square)$. We treat $A$
by the following cases.

\begin{enumerate}
\item (A1) $A$ is atomic. Then $(A^\square)^+=A^\square$, by definition,
and  result holds trivially.
\item (A1-A4, B1) $ A=\square B , A=B\wedge C, A=B\vee C$. All these cases hold
by induction hypothesis. In boxed case, we use of induction
hypothesis and A4. In conjunction, we use of A1-A3 and in
disjunction we use A1,A2 and B1.
\item $ A=B \ra C $. There are several sub-cases. similar to
definition of $\NNIL$-algorithm, an occurrence of a sub-formula $B$ of $A$ is
said to be an {\em outer occurrence} in $A$, if it is neither in
the scope of a $\Box$ nor in the scope of $\ra$.

(c).i.(A1-A3) $C$ contains an outer occurrence of a conjunction.
We can treat this case using induction hypothesis and
$\TNNIL$-algorithm.

(c).ii.(A1-A3) $B$ contains an outer occurrence of a disjunction.
We can treat this case by induction hypothesis and
$\TNNIL$-algorithm.

(c).iii. $B=\bigwedge X$ and $C=\bigvee Y$, where $X, Y$ are sets
of implications, atoms and boxed formulae. We have several
sub-cases.

(c).iii.$\alpha$.(A1-A4, B3)  $X$ contains atomic variables. Let $p$ be an atomic
variable in $X$. Set $D:=\bigwedge(X\setminus\{p\})$. Then
\begin{align*}
(A^\square)^+& \equiv\boxdot[(p\wedge\square p)\ra(D^\square\ra
C^\square)^+]\\
 &\equiv \boxdot[(p\wedge\square p)\ra((D\ra C)^\square)^+]
\end{align*}
On the other hand, we have by induction hypothesis and A1,A2 and B3, that
$$ [(p\wedge\square p)\ra((D\ra C)^\square)^+]\blrt^* (p\wedge\square p)\ra((D\ra C)^\square)$$
which by use of A4 implies:
$$\bo[(p\wedge\square p)\ra((D\ra C)^\square)^+]\blrt^* \bo[(p\wedge\square p)\ra((D\ra C)^\square)]$$
And by use of A1-A3 we have
$$\boxdot[(p\wedge\square p)\ra((D\ra C)^\square)^+]\blrt^* \boxdot[(p\wedge\square p)\ra((D\ra C)^\square)]$$
Finally by A1 and A2 we have : $(A^\square)^+\blrt^* A^\square$.

(c).iii.$\beta$.(A1-A4, B3) $X$ contains boxed formula. Similar
to the previous case.

(c).iii.$\gamma$.(A1, A2) $X$ contains $\top$ or $\bot$. Trivial.

(c).iii.$\delta$.(A1-A4, B2, B3) $X$ contains only implications. This case needs the
axiom B2 and it seems to be the interesting case.

By Lemma \ref{Lemma-technical property},
$$(A^\square)^+\equiv\boxdot\left(\bigwedge_{E\ra F\in
X}\square\left(\left(E\ra
F\right)^\square\right)^+\ra\left(\bigwedge\left\lbrace\left(\left(B\downarrow
D\ra C\right)^\square\right)^+ \mid D\in
X\right\rbrace\wedge\left(\left([B]Z\right)^\square\right)^+\right)
\right).$$
Then by induction hypothesis, A1-A4 and B3 we have:
\begin{align*}
(A^\square)^+&\blrt^*\boxdot\left(\bigwedge_{E\ra F\in
X}\square\left(E\ra F\right)^
\square\ra\left(\bigwedge\left\lbrace\left(B\downarrow D\ra
C\right)^\square \mid D\in X
\right\rbrace\wedge\left([B]Z\right)^\square\right) \right)\\
& \blrt^*\left( \square B\ra \left(\bigwedge\left\lbrace
B\downarrow D\ra C \mid D\in X \right\rbrace\wedge[B]Z\right)
\right)^\square
\end{align*}

We show that for each $E\in Z$,
\begin{center}
(*)\hspace{.5in}  $i\K4\vdash(\bigwedge\{(B\downarrow D)\ra C\mid
D\in X\}\wedge [B]E)\ra A.$
\end{center}
If $E=C$, we are done by $\IPC_\square\vdash[B]C\ra(B\ra C)$. So
suppose some $E\ra F\in X$. We reason in $i\K4$. Assume
$\bigwedge\{(B\downarrow D\ra C\mid D\in X\}$, $[B]E$ and $B$. We
want to derive $C$. We have $(\bigwedge(X\setminus\{E\ra
F\})\wedge F)\ra C$, $[B]E$ and $B$. From $B$ and $[B]E$, we
derive $E$. Also from $B$, we derive $E\ra F$, and so $F$. Hence
we have $\bigwedge(X\setminus\{E\ra F\})\wedge F$, which implies
$C$, as desired.

Now (*) implies
$$i\K4\vdash\overbrace{(\bigwedge\{(B\downarrow D\ra C\mid D\in X\}\wedge
[B]Z)}^G\ra A$$

Then by Proposition \ref{Proposition-propositional properties of
Box translation}, we have $i\K4\vdash (G^\square\wedge B^\bo)\ra
C^\bo$. This implies $i\K4\vdash (B^\bo\ra(G^\square\wedge
B^\bo))\ra(B^\bo\ra C^\bo) $, and hence  $i\K4\vdash (B^\bo\ra
G^\bo)\ra (B^\bo\ra C^\bo)$. Then because $B^\bo\ra \bo B^\bo$,
we have $i\K4\vdash (\bo(B^\bo)\ra G^\bo)\ra (B^\bo\ra C^\bo)$.
Hence by necessitation, we derive $i\K4\vdash(\square B\ra
(\bigwedge\{(B\downarrow D\ra C\mid D\in
X\}\wedge[B]Z))^\square\ra A^\square$. Hence $(A^\square)^+\brt^*
A^\square$.

To show the other way around, i.e., $A^\square\brt^* (A^\square)^+$,
by Proposition \ref{Proposition-propositional properties of Box
translation}, it is enough to show
$$ A\brt^* \left(\square B\ra \left(\bigwedge\left\lbrace
B\downarrow D\ra C \mid D\in X\right\rbrace\wedge[B]Z\right)
\right)$$
or equivalently
$$A\wedge\bo B\brt^*  \left(\bigwedge\left\lbrace
B\downarrow D\ra C \mid D\in X\right\rbrace\wedge[B]Z \right)$$
We have $\IPC_\bo\vdash A\ra \bigwedge\left\lbrace B\downarrow
D\ra C \mid D\in X\right\rbrace$, and hence by A1, $A\wedge\bo
B\brt^* \bigwedge\left\lbrace B\downarrow D\ra C \mid D\in
X\right\rbrace $. On the other hand, $A\wedge\bo B\brt^* [B]Z$,
which by A3, implies
$$A\wedge\bo B\brt^*  \left(\bigwedge\left\lbrace B\downarrow D\ra C \mid D\in X\right\rbrace\wedge[B]Z\right)$$
\end{enumerate}
\end{proof}

\section{The $\Sigma_1$-Provability Logic of $\HA^*$}
In this section we will show that $\LCS$ is the provability logic
of $\HA^*$ for $\Sigma_1$-substitutions.

Before we continue with the soundness and completeness theorem,
let us state the main theorem from \cite{Sigma.Prov.HA}   that plays a crucial 
role in the sequel of this paper. 
\begin{theorem}\label{Theorem-Main tool}
	Let  
	  $A\in \TNNIL^\Box$ be a modal proposition such that $\LC\nvdash A$.
	   Then  there exists some arithmetical $\Sigma_1$-substitution 
	$\sigma$  such that  $ \HA\nvdash \sigma_{_{\sf HA}}(A)$.
\end{theorem}
\begin{proof}
For the  rather long  proof of this fact, see 
\cite{Sigma.Prov.HA}, Theorems 4.26 and  5.1.
\end{proof}

\subsection{The Soundness Theorem}
Let us define some notions from \cite{Visser02}. We call a
first-order sentence $A$, $\Sigma$-preserves $B$
($A\rhd_{T,\Sigma_1}B$), if for each $\Sigma_1$-sentence $C$, if
$T\vdash C\ra A$, then $T\vdash C\ra B$. For modal propositions
$A$ and $B$, we define $A\rhd_{T,\Sigma_1,\Sigma_1}B$  iff for
each arithmetical $\Sigma$-substitution $\sigma_T$, we have
$\sigma_T(A)\rhd_{T,\Sigma_1}\sigma_T(B)$. For arbitrary modal
sentences $A, B$, the notation $A\ar_{T,\Sigma_1}B$ means that
$T\vdash\sigma_T(A)$ implies $T\vdash\sigma_T(B)$, for arbitrary
$\Sigma_1$-substitution $\sigma_T$. All the above relations with a
superscript of $HA$, means  ``an arithmetical formalization of
that relation in \HA", for example, $A\rhd^{\sf HA}_{{\sf
HA}^*,\Sigma_1}B$ means $\HA\vdash \text{``}A\rhd_{{\sf
HA}^*,\Sigma_1}B\text{''} $.

\begin{lemma}\label{Lemma-Sigma preservativity}
\begin{enumerate}
\item For each first-order sentences $A,B$, $A\rhd^{{\sf
HA}}_{{\sf HA}^*,\Sigma_1}B$ iff $A^\square\rhd^{{\sf HA}}_{{\sf
HA},\Sigma_1}B^\square$,
\item For each propositional modal $A,B$, $A\rhd^{{\sf
HA}}_{{\sf HA}^*,\Sigma_1,\Sigma_1}B$ iff $A^\square\rhd^{{\sf
HA}}_{{\sf HA},\Sigma_1,\Sigma_1}B^\square$.
\end{enumerate}
\end{lemma}
\begin{proof}
To prove part 1, use Lemma 
\ref{Lemma-Properties of Box translation}.1 and definitions of $\rhd^{{\sf HA}}_{{\sf
HA}^*,\Sigma_1}$ and $\rhd^{{\sf HA}}_{{\sf HA},\Sigma_1}$.

To prove part 2, note that $A\rhd^{{\sf HA}}_{{\sf
HA}^*,\Sigma_1,\Sigma_1}B$ iff for all $\Sigma$-substitution
$\sigma$, $\sigma_{{\sf HA}^*}(A)\rhd^{{\sf HA}}_{{\sf
HA}^*,\Sigma_1}\sigma_{{\sf HA}^*}(B)$ iff  for all
$\Sigma$-substitution $\sigma$, $\sigma_{{\sf
HA}^*}(A)^\square\rhd^{{\sf HA}}_{{\sf HA},\Sigma_1}\sigma_{{\sf
HA}^*}(B)^\square$ (by previous part) iff for all
$\Sigma$-substitution $\sigma$, $\sigma_{{\sf
HA}}(A^\square)\rhd^{{\sf HA}}_{{\sf HA},\Sigma_1}\sigma_{{\sf
HA}}(B^\square)$ iff $A^\square\rhd^{{\sf HA}}_{{\sf
HA},\Sigma_1,\Sigma_1}B^\square$.
\end{proof}

\begin{lemma}\label{Lemma-B_1}
$\rhd^{{\sf HA}}_{{\sf HA},\Sigma_1}$ is closed under $B1$.
\end{lemma}
\begin{proof}
See \cite{Visser02}, 9.1.
\end{proof}

\begin{corollary}\label{Corollary-B_1}
$\rhd^{{\sf HA}}_{{\sf HA}^*,\Sigma_1}$ is closed under $B1$.
\end{corollary}
\begin{proof}
Immediate corollary of Lemma \ref{Lemma-Sigma preservativity} and \ref{Lemma-B_1}.
\end{proof}

\begin{lemma}\label{Lemma-B_2}
$\rhd^{{\sf HA}}_{{\sf HA},\Sigma_1,\Sigma_1}$ satisfies
A$1$-A$4$, B$1$,  $\text{B}2'$ and B$3$.
\end{lemma}
\begin{proof}
Proof of closure under A1-A4 and B3 is straightforward. Closure
under B1 is by Lemma \ref{Lemma-B_1}. For a proof of case
$\text{B}2'$, see \cite{Visser02}.9.2.
\end{proof}

\begin{corollary}\label{Corollary-B_2}
$\rhd^{{\sf HA}}_{{\sf HA}^*,\Sigma_1,\Sigma_1}$ satisfies $B2$.
\end{corollary}
\begin{proof}
Let $A,B,C,X,Z$ be as stated in defining B2. We must prove
$A\wedge\square B\rhd^{{\sf HA}}_{{\sf
HA}^*,\Sigma_1,\Sigma_1}[B]Z$. Hence by Lemma \ref{Lemma-Sigma
preservativity}, it is enough to show $(A\wedge\square
B)^\square\rhd^{{\sf HA}}_{{\sf
HA},\Sigma_1,\Sigma_1}([B]Z)^\square$. Let $X':=\{E^\square\ra
F^\square | E\ra F\in X\}, B':=\bigwedge X', C':=C^\square,
Z':=\{E^\square | E\ra F\in X\}\cup\{C'\}$.  Now Because
$\rhd^{{\sf HA}}_{{\sf HA},\Sigma_1,\Sigma_1}$ satisfies B2$'$
(Lemma \ref{Lemma-B_2}), we have $(B'\ra C')\rhd^{{\sf HA}}_{{\sf
HA},\Sigma_1,\Sigma_1}[B']Z'$. Note that $Z^\square=Z'$ and
$\IPC_\square\vdash (B'\wedge\square B')\lr B^\square$. Hence by
Lemma \ref{Lemma-property for bracket}.2, $i\K4+\square
B'\vdash[B']Z'\lr[B^\square]Z^\square$. Also by Lemma
\ref{Lemma-property for bracket}.1, $i\K4+\square B'\vdash
[B^\square]Z^\square\lr([B]Z)^\square$. So $i\K4+\square B'\vdash
[B']Z'\lr ([B]Z)^\square$. Now because $\rhd^{{\sf HA}}_{{\sf
HA},\Sigma_1,\Sigma_1}$ satisfies A$1$, we have $\square
B'\rhd^{{\sf HA}}_{{\sf HA},\Sigma_1,\Sigma_1} [B']Z'\lr
([B]Z)^\square$. Now one can easily observe that because
$\rhd^{{\sf HA}}_{{\sf HA},\Sigma_1,\Sigma_1}$ is closed under
A$1$-A$3$, we can deduce $(B'\ra C')\wedge\square B'\rhd^{{\sf
HA}}_{{\sf HA},\Sigma_1,\Sigma_1} ([B]Z)^\square$. This by using
A$1$-A$3$ implies $((B\ra C)\wedge\square B)^\square\rhd^{{\sf
HA}}_{{\sf HA},\Sigma_1,\Sigma_1} ([B]Z)^\square$. Hence by Lemma
\ref{Lemma-Sigma preservativity}.2, $(B\ra C)\wedge\square
B\rhd^{{\sf HA}}_{{\sf HA},\Sigma_1,\Sigma_1} [B]Z$, as desired.
\end{proof}

\begin{corollary}\label{Corollary-B_3}
$\rhd^{{\sf HA}}_{{\sf HA}^*,\Sigma_1,\Sigma_1} $ is closed under $B3$.
\end{corollary}
\begin{proof}
Let $p$ be atomic or boxed and assume some $ A, B$ such that
$A\rhd^{{\sf HA}}_{{\sf HA}^*,\Sigma_1,\Sigma_1}B$. Then by Lemma
\ref{Lemma-Sigma preservativity}.2, $A^\square\rhd^{{\sf
HA}}_{{\sf HA},\Sigma_1,\Sigma_1}B^\square$. Because $\rhd^{{\sf
HA}}_{{\sf HA},\Sigma_1,\Sigma_1}$ satisfies B3, we get
$p^\square\ra A^\square\rhd^{{\sf HA}}_{{\sf
HA},\Sigma_1,\Sigma_1}p^\square\ra B^\square$. Now by A4,
$\square[p^\square\ra A^\square]\rhd^{{\sf HA}}_{{\sf
HA},\Sigma_1,\Sigma_1}\square[p^\square\ra B^\square]$, which
implies $(p\ra A)^\square\rhd^{{\sf HA}}_{{\sf
HA},\Sigma_1,\Sigma_1}(p\ra B)^\square$. Now by Lemma
\ref{Lemma-Sigma preservativity}.2, $p\ra A\rhd^{{\sf HA}}_{{\sf
HA}^*,\Sigma_1,\Sigma_1}p\ra B$, as desired.
\end{proof}

\begin{lemma}\label{Lemma-Admissible rules of HA^*}
We have the following inclusions:
$$ \brt^*\ \subseteq\, \rhd^{{\sf HA}}_{{\sf HA}^*,\Sigma_1,\Sigma_1} \subseteq\ \ar^{{\sf HA}}_{{\sf HA}^*,\Sigma_1}$$
\end{lemma}
\begin{proof}
The second inclusion is a trivial. We only prove
the first inclusion. We show that $\rhd^{{\sf HA}}_{{\sf
HA}^*,\Sigma_1,\Sigma_1}$ is closed under A1-A4 and
B1-B3. One can observe that $\rhd^{{\sf HA}}_{{\sf
HA}^*,\Sigma_1,\Sigma_1}$ is closed under A1-A4 and we
leave this to the reader. Closure under B1, B2 and B3 is by Corollaries
\ref{Corollary-B_1},\ref{Corollary-B_2} and \ref{Corollary-B_3}, respectively.
\end{proof}

\begin{theorem}\label{Theorem-Soundness for Sigma substitutions PL}
\textup{(\textbf{Soundness})}
$\LCS$ is sound for $\Sigma_1$-arithmetical interpretations in
$\HA^*$, i.e. $\LCS\subseteq\PLS(\HA^*)$.
\end{theorem}
\begin{proof}
We show that for arbitrary $\Sigma$-substitution, $\sigma_{{\sf
HA}^*}$, and for any $A$, if $\LCS\vdash A$, then
$\HA^*\vdash\sigma_{{\sf HA}^*}(A)$. This can be done by
induction on the complexity of  $\LCS\vdash A$. 
All inductive steps clearly holds, except for the axioms $\square A\ra
\square B$ with $A\brt^* B$. This case is a direct consequence of
Lemma \ref{Lemma-Admissible rules of HA^*}.
\end{proof}

\subsection{The Completeness Theorem}
\begin{theorem}\label{Theorem-Sigma_1 Completeness for provability logic}
$\Sigma_1$-arithmetical interpretations in $\HA^*$  are complete
for $\LCS$, i.e. $$\PLS(\HA^*)\subseteq\LCS$$
\end{theorem}
\begin{proof}
We prove the Completeness Theorem contra-positively. Let
$\LCS\nvdash A(p_1,\ldots,p_n)$. Then $\LCS\nvdash A^\square$ and
hence by Corollary \ref{Corollary-+algorithm}, $\LCS\nvdash
(A^\square)^-$. This, by  
\Cref{Theorem-NNIL Crucial Properties}, implies $\LCS\nvdash ((A^\square)^-)^*$ 
and hence
$\LCS\nvdash (A^\square)^+$, and a fortiori, $\LC\nvdash
(A^\square)^+$. Hence
by Theorem \ref{Theorem-Main tool}, there exists some
$\Sigma_1$-substitution $\sigma$, such that
$\HA\nvdash\sigma_{{\sf HA}}((A^\square)^+)$. Hence by 
\Cref{Lemma-NNIL properties}.1, 
$\HA\nvdash\sigma_{{\sf
HA}}(A^\square)$
 and by Lemma 
\ref{Lemma-Properties of Box translation 2}, 
$\HA^*\nvdash\sigma_{{\sf HA}^*}(A)$.
\end{proof}

\begin{corollary}\label{Corollary-lles-LCS}
For any modal proposition $A$, 
$\LCS\vdash A$ iff $\lles\vdash A^\Box$.
\end{corollary}
\begin{proof}
By Theorems \ref{Theorem-Soundness for Sigma substitutions PL} and \ref{Theorem-Sigma_1 Completeness for provability logic} and \Cref{Lemma-Properties of Box translation 2}.
\end{proof}

\begin{corollary}\label{Corollary-decidability-of-LCS}
$\LCS$ is  decidable. 
\end{corollary}
\begin{proof}
A proof can be given either with inspections in the proof of the Completeness Theorem (\ref{Theorem-Sigma_1 Completeness for provability logic}) or by using the decidability of 
$\lles$ \cite{Sigma.Prov.HA} and \Cref{Corollary-lles-LCS}.
\end{proof}
\subsection*{Open problems}
\begin{enumerate}
\item  
The statement of 
\Cref{Corollary-lles-LCS} is {\em purely propositional}. However, our proof of this corollary is based on 
\Cref{Theorem-Sigma_1 Completeness for provability logic}, that has  {\em arithmetical} theme. A  tempting problem is 
to find  a {\em direct propositional proof} for this corollary. Then we can derive \Cref{Theorem-Sigma_1 Completeness for provability logic}.

\item We conjecture that the full provability logic of $\HA^*$ is the logic ${\sf iH}^*$, 
axiomatized as follows  
$${\sf iH}^*:=\iGL+\CP+\{ \Box A\to \Box B : A\brt_\alpha^* B\},$$
in which the relation $\brt^*_\alpha$ is defined as the smallest relation satisfying:
\begin{itemize}
 \item A1. If $i\K4\vdash A\ra B$, then $A\brt_\alpha^* B$,
 \item A2. If $A\brt_\alpha^* B$ and $B\brt_\alpha^* C$, then $A\brt_\alpha^* C$,
 \item A3. If $C\brt_\alpha^* A$ and $C\brt_\alpha^* B$, then $C\brt_\alpha^* A\wedge B$,
 \item A4. If $A\brt_\alpha^* B$, then $\bo A\brt_\alpha^* \bo B$,
 \item B1. If $A\brt_\alpha^* C$ and $B\brt_\alpha^* C$, then $A\vee B\brt_\alpha^* C$,
 \item B2. Let $X$ be a set of implications, $B:=\bigwedge X$ and $A:=B\ra C$.
 Also assume that $Z:=\{E | E\ra F\in X\}\cup \{C\}$. 
 Then $A\wedge\square B\brt_\alpha^* \{B\}Z $,
 \item B3. If $A\brt_\alpha^* B$, then $\Box C\ra A\brt_\alpha^* \Box C\ra B$.
\end{itemize}
The notation $ \{A\} (B)$, for modal propositions $A$ and $B$, 
  is defined inductively:
\begin{itemize}
\item $\{A\}(\Box B)=\Box B$ and $\{A\}(\bot)=\bot$.
\item $\{A\}(B_1\circ B_2) = \{A\}(B_1)\circ \{A\}(B_2)$, 
for $\circ\in\{\vee,\wedge\}$,
\item $\{A\}(B) = A\ra B$ for all of the other cases, i.e. when $B$ is atomic variable or implication.
\end{itemize}
And  $\{A\}\Gamma$,  for a set $\Gamma$ of modal propositions, is defined as 
$\bigvee_{B\in \Gamma} \{A\}(B)$.
\end{enumerate}

\providecommand{\bysame}{\leavevmode\hbox to3em{\hrulefill}\thinspace}
\providecommand{\MR}{\relax\ifhmode\unskip\space\fi MR }
\providecommand{\MRhref}[2]{%
  \href{http://www.ams.org/mathscinet-getitem?mr=#1}{#2}
}
\providecommand{\href}[2]{#2}

\end{document}